\documentclass{proc-l}
\usepackage{mathtools,amsfonts,amssymb,amsthm,mathrsfs}
\usepackage{lmodern}
\usepackage{graphicx}
\usepackage{subcaption}
\usepackage{tikz}
\usepackage{pstricks}
\usepackage{tabularx}
\usepackage{enumitem}
\usepackage[english]{babel}
\usepackage[T1]{fontenc}
\usepackage[utf8]{inputenc}
\usepackage{comment}
\usepackage[normalem]{ulem}
\usepackage{dsfont} 

\usepackage{chngcntr} 

\counterwithin*{equation}{section}

\theoremstyle{plain}

\newtheorem{theorem}{Theorem}[section]
\newtheorem{lemma}[theorem]{Lemma}
\newtheorem{proposition}[theorem]{Proposition}

\theoremstyle{definition}

\newtheorem{definition}{Definition}[section]
\newtheorem{remark}[theorem]{Remark}
\newtheorem{assumption}{Assumption}[section]

\usepackage{hyperref}
\hypersetup{colorlinks,%
	citecolor=red,%
	filecolor=yellow,%
	linkcolor=black,%
	urlcolor=black}

\newcommand{\unmu}{\underline{\mu}}

\newcommand \eps {\varepsilon}

\newcommand \R {\mathbb R}

\renewcommand{\L}{\mathrm{L}} 

\newcommand{\BC}{\mathrm{BC}}
\newcommand{\QQ}{\mathcal{Q}}
\newcommand{\CC}{\mathcal{C}}

\newcommand \FF {\mathcal{F}}

\newcommand \LL {\mathcal{L}}  

\newcommand{\dd}{\mathrm{d}}

\renewcommand{\Re}{\mathrm{Re}}

\begin{document}
	
	\title{Bistable wavefronts in the Gurtin-MacCamy population model}
    
	\author{Quentin Griette}
	
	\address{
		Université Le Havre Normandie, Normandie Univ.,
		LMAH UR 3821, F-76600 Le Havre, France. 
	}
	\curraddr{}
	\email{quentin.griette@univ-lehavre.fr}
	\thanks{}
	
	\author{Franco Herrera}
	\address{
		Instituto de Matemáticas, Universidad de Talca, Talca, Chile; 
		Université Le Havre Normandie, Normandie Univ.,
		LMAH UR 3821, F-76600 Le Havre, France. 
	}
	\curraddr{}
	\email{franco.herrera@utalca.cl}
	\thanks{}
	\author{Hao Kang}
	
	\address{
		Center for Applied Mathematics, Tianjin University, Tianjin 300072, China; 
	}
	\curraddr{}
	\email{haokang@tju.edu.cn}
	\thanks{}

	\commby{}
	
	\begin{abstract}
		In this paper we prove the existence of monotone traveling wave solutions for an age-structured population model with a monotone and bistable nonlinearity. We formulate the problem in terms of a discrete-time recurrence $u_{n+1} = Q[u_n]$, with operator $Q$ encoding the role of the speed into its definition. For each speed, we prove the existence of a traveling wave by the general theory of monotone evolution systems, and from appropriate estimates on the speed, we deduce the existence of a traveling wave of the original problem. Also, the sign of the speed is established under an invasion condition, and the asymptotic behavior of the profile is determined by some Tauberian results.
	\end{abstract}
	\maketitle
	
	\section{Introduction}\label{sec:intro}
	
	Since the pioneering works of Fisher and of Kolmogorov, Petrovski and Piskunov on the spatial spread of an advantageous gene, {the study of propagation phenomena has attracted a lot of attention, and} traveling wave solutions have become a fundamental tool for describing biological invasions. A traveling front represents an interface separating regions in different population states, while its speed quantifies the rate and direction at which one state invades the other. Alternatively, in epidemiology, traveling waves can describe the advance of an epidemic from an initially localized infected region; see, e.g., the monograph of Murray \cite{MR1952568}, the survey of Ruan \cite{MR2309365}, and the works \cite{DenuNgomaSalako2020,DenuNgomaSalako2023} for related reaction-diffusion epidemic models.
	
	At the same time, a realistic description of population dynamics often requires keeping track of individual heterogeneity. Age, size, developmental stage, or infection age may strongly affect mortality, reproduction, transmission, and even spatial movement. Structured population models incorporate such individual-level information into population-level dynamics and therefore provide a natural framework in which demographic or epidemiological processes can be coupled with spatial dispersal. The interaction between structure and space has motivated a very rich literature; see, among others, \cite{MR2737849,MR2842101,MR2564476,MR1852431} and, more recently, \cite{MR4929042,MR4984351}.
	
	Most of these propagation results concern to the so-called monostable case, in which the extinction state is unstable and a positive state invades whenever the population is successfully introduced. On the other hand, in bistable dynamics, both extinction and a positive equilibrium are locally stable and are separated by an unstable state. Such a structure naturally arises in the presence of strong Allee effects: a population introduced below the threshold becomes extinct, whereas sufficiently large ones may establish and invade. Consequently, the existence and direction of a bistable front encode the competition between two locally stable population states. {Even without spatial diffusion, such threshold dynamics arise in age-structured models, e.g.,} \cite{Griette2026}. Our purpose is to understand how bistability interacts with age structure and spatial diffusion.
	
	More precisely, we consider the age-structured population model introduced by Gurtin and MacCamy in \cite{MR354068},
	\begin{equation}\label{GM-model}
		\begin{dcases}
			(\partial_t + \partial_a - \partial_{xx}) u(t,a,x) = -\mu(a) u(t,a,x), & t\in \R,~a>0,~x\in\R \\
			u(t,0,x) = f\left( \int_0^{+\infty} \beta(a) u(t,a,x) \dd a \right) & t\in\R,~x\in \R.
		\end{dcases}
	\end{equation}
	
	Here $u(t,a,x)$ denotes the density of individuals of age $a$ at time $t$ and position $x$, while $\mu$ and $\beta$ are the age-dependent mortality and fertility rates, respectively. The nonlinear birth function $f$ accounts for density-dependent reproductive effects and, in the present work, is assumed to possess a bistable structure.
	
	Age-structured models of this type have received increasing attention in the recent years. For the Gurtin-MacCamy equation, Ducrot and Magal \cite{MR3988617} established the existence of wave trains in multidimensional space employing a Hopf bifurcation argument for the associated non-densely defined Cauchy problem. Recently, Ducrot and Kang \cite{MR4929042} proved the existence of traveling waves and spreading speeds for an age-structured epidemic model with nonlocal diffusion. Similarly, Kang and Liu \cite{MR4984351} investigated the global dynamics and spreading properties for a diffusive age-structured population model in a spatially periodic environment. As before, these works primarily deal with monostable propagation mechanisms. To the best of our knowledge, the existence of traveling fronts for the Gurtin-MacCamy model in the bistable regime remains open.
	
	{We also focus this study on the monotone scenario}. A major difficulty comes from the order structure of the equation. Although the semiflow generated by the model is monotone, it does not enjoy the stronger order properties that are often available for scalar parabolic equations, such as strong monotonicity or eventual strong order preservation; see \cite{Griette2026}. We therefore approach the traveling-wave problem indirectly. For each $c\in\R$, the characteristic representation of \eqref{GM-model} allows us to reformulate the profile equation as a nonlinear convolution equation where the convolution kernel $G_c$ incorporates both the age structure and the prospective wave speed $c$. Thus, traveling fronts of the original model with speed $c$ are precisely the nonconstant fixed points of $Q_c$.
	
	Thus, rather than searching immediately for a fixed point of $Q_c$, we can regard it as the time-one map of a discrete monotone dynamical system. The abstract theory of Fang and Zhao \cite{MR3420507} then yields a bistable traveling wave for the recursion generated by $Q_c$, with some speed $v(c)$. Then, from the squeezing argument given by Chen in \cite{MR1424765}, we show that speed is uniquely determined and depends continuously on $c$. Some estimates on $v(c)$ follow from the spreading properties of the two associated monostable subsystems, in the spirit of the comparison theory developed in \cite{LUI1989269,MR2270161}, and showing that its sign changes as $c$ varies. We hence deduce the existence of a standing wave for the corresponding map $Q_{c^*}$, as we desired. Finally, adapting the convolution techniques of Diekmann and Kaper \cite{MR512163}, we obtain a precise description of the exponential asymptotic behavior of the wave near the two stable equilibria.
	
	\subsection{Assumptions and main results}
	
	\begin{definition}
		A \emph{traveling wave solution} of \eqref{GM-model} is a pair $(c,U)\in \R \times C(\R,\L^1(0,+\infty))$ such that $u(t,a,x) = U(a,x-ct)$ is an entire solution of \eqref{GM-model}. $c$ is called the \emph{speed} of the wave and $U$ the \emph{wave profile}. Additionally, we say that $U$ connects the stable equilibria in the system if
		\begin{equation}\label{eq:lim}
			U(a,-\infty) = \Pi(a):=e^{-\int_0^a\mu(s)ds} \quad \text{and}\quad U(a,+\infty) = 0.
		\end{equation}
		In particular, if $U$ is a monotonic wave profile in space that connects the stable equilibria, we simply call $U$ a \text{bistable wavefront}.
	\end{definition}
	
	The notation $\BC(\R)$ stands for the space of real-valued bounded and continuous functions on $\R$, and it is used $\|\cdot\|$ to denote the $\sup$-norm in this space. This space is also an ordered Banach space with the order defined by the cone of bounded nonnegative and continuous functions $\BC_+(\R)$. For $u,v\in \BC(\R)$ we write $u\leq v$ provided that $v-u\in \BC_+(\R)$, $u<v$ provided $u\leq v$ but $u\not\equiv v$. The relations $u\geq(>) v$ are analogously defined. Clearly, every number $\alpha\in \R$ can be regarded as an element of $\BC(\R)$. Also, the following notation will be used for $\alpha\in \R$
	\begin{equation}\label{C_alpha}
		\CC_\alpha = \{ u\in\BC(\R) \colon 0 \leq u \leq \alpha \}.
	\end{equation}
	
	This study is devoted to the bistable case for the nonlinearity $f$ in \eqref{GM-model}. Mainly, we will assume that the following statements hold throughout the paper.
	
	\begin{assumption}\label{as-1}
		The following conditions over the parameters in \eqref{GM-model} hold:
		\begin{enumerate}[label = (\roman*)]
			\item {$\beta$ and $\mu$ are nonnegative measurable functions for which there exist constants $\beta^*>0$ and $\mu^*>\underline{\mu}>0$ such that $0\leq \beta(a)<\beta^*$ and $0<\underline{\mu}\leq \mu(a) \leq \mu^*$} for almost every $a\geq 0$. Moreover, they are normalized as
			\begin{equation}\label{eq:norm}
				\int_0^{+\infty} \beta(a) e^{-\int_0^a \mu(l)\dd l}\dd a = 1.
			\end{equation}
			\item $f\colon [0,1] \to \R$ is a strictly increasing $C^1$-function with bounded derivative. It has a  bistable structure in the sense that the equation $f(x)=x$ has exactly three solutions: $0<\theta<1$; 
			{finally, $0<f'(0),f'(1)<1$ and $f'(\theta)>1$.}
		\end{enumerate}
	\end{assumption}
	
	Hereinafter we set $\Pi(a)$ for the survival probability of an individual to age $a$, i.e., $\Pi(a):= e^{-\int_0^a \mu(l) \dd l}$. We further establish the notation
	\begin{equation}\label{eq:laplace_H}
		H(\lambda) := \int_0^{+\infty} \beta(a)\Pi(a)e^{\lambda a} \dd a,
	\end{equation}
	where the integral on the right is finite for every $\lambda\in[0,\Lambda^*)$, with $\Lambda^*\in(0,+\infty]$ being the abscissa of convergence. For some estimates for this value, see Proposition \ref{prop:esti_lambda}. Also, $f$ may be extended to the whole real line by linear functions matching its slope at $0$ and $1$, and hence we assume $f\in C^1(\R)$ when it is needed.
	
	
	To study the asymptotic behavior of the traveling wave solutions, more regularity on the function $f$ is needed, as well as some growing conditions for the exponential moments of $\beta(a)\Pi(a)$. Mainly, we impose the additional hypotheses:
	
	\begin{assumption} \label{as-2}
		The following properties for $f$ and $H$ hold:
		\begin{enumerate}[label=(\roman*)]
			\item $f$ is of class $C^{1+\varepsilon}$  for some $\eps\in(0,1)$ at the points $x=0$ and $x=1$.
			\item $H(\lambda)$ has a pole at the abscissa $\Lambda^*$, i.e., $H(\lambda) \to +\infty$ as $\lambda\uparrow \Lambda^*$.
		\end{enumerate}
	\end{assumption}
	
	The first part of Assumption \ref{as-2} means that
	\begin{equation}\label{eq:holder}
		f(x) = f'(0)x + O(x^{1+\eps}) \; \text{as}~ x\downarrow 0, \; \text{and} \; 1-f(1-x) = f'(1)x + O(x^{1+\eps}) \; \text{as}~ x\downarrow 0.
	\end{equation}
	
	\begin{remark}
		Condition (ii) above guarantees that the function $H$ attains every positive real number in $(-\infty,\Lambda^*)$. This condition is easily verified in several biological settings. For example, if the maximal age of reproduction of the population is finite, i.e., when $\beta$ has compact support; as well as when the age-dependent fertility and death rates are eventually constant. Notwithstanding, some particular choices for these parameters may conduce to the failure of such a condition; for instance, if $\beta(a)= \frac{C}{(1+a)^2}$ and $\mu$ is age-independent, with $C$ an appropriate constant to fulfill the normalization hypothesis, then it holds that $H(\lambda)\leq 1$ for all $\lambda \leq \mu$ but the integral diverges for $\lambda > \mu$.
	\end{remark}
	
	Now we can state the main results of this study. The first one is concerned about the existence of bistable wavefronts for the Gurtin-MacCamy model.
	
	\begin{theorem}\label{teo:main_1}
		Let Assumption \ref{as-1} hold. Then, there exist a pair $(c,U)\in \R \times C(\R,\L^1(0,+\infty))$ such that $u(t,a,x)= U(a,x-ct)$ defines a bistable wavefront of \eqref{GM-model} with speed $c$. Moreover, any traveling wave solution connecting the stable equilibria of the model is strictly decreasing in the space variable.
	\end{theorem}
	
	The next theorem collects some results about the asymptotic behavior of the traveling wave solutions for the system \eqref{GM-model} as well as the sign of the speed under an invasion condition commonly found in the literature.
	
	\begin{theorem}\label{teo:main_2}
		Let Assumption \ref{as-1} be satisfied. Then we have the following:
		\begin{enumerate}[label=\roman*)]
			\item If the invasion condition
			\[
			\int_0^1 (f(x)- x) \dd x>0
			\]
			holds, then the speed of any bistable wavefront is positive.
			\item If Assumption \ref{as-2} is also satisfied and $(c,U=U(a,x))$ is any given bistable wavefront with speed $c$, then there exist positive constants $C_0$ and $C_1$ for which
			\begin{equation}\label{eq:asymp}
				\begin{array}{>{\displaystyle}c}
					U(a,x) \sim C_0 f'(0) \Pi(a) e^{a(\lambda_0^2 - \lambda_0 c)} e^{-\lambda_0 x} \quad \text{as}~ x\to +\infty, \\[10pt]
					\Pi(a) - U(a,x) \sim C_1 f'(1) \Pi(a) e^{a(\lambda_1^2 + \lambda_1 c)} e^{\lambda_1 x} \quad \text{as}~ x\to -\infty, 
				\end{array}
			\end{equation}
			where $\lambda_0$ and $\lambda_1$ are respectively defined as the unique positive solutions of the equations
			\[
			1-f'(0)\int_\R e^{\lambda x} G_c(x) \dd x = 0 \quad \text{and}\quad 1-f'(1)\int_\R e^{-\lambda x} G_c(x) \dd x =0,
			\]
			where $G_c$ is defined in \eqref{eq:Gc}.
		\end{enumerate}
	\end{theorem}
	
	The article is structured as follows. In Section \ref{sec:prelim} {we present an integral formulation for which the traveling wave solutions to the Gurtin-MacCamy model with a given speed $c$ coincide with the fixed points of a given nonlinear convolution operator $Q_c$.} Such a problem fits naturally in the general theory of bistable traveling waves in abstract monotone evolution systems. The {key part} is the counter-propagation condition, which we infer thanks to the so-called Freidlin-Gärtner formula as a manner to estimate the spreading speed of the monostable subsystems involved in the model. Section \ref{sec:existence} is devoted to the existence and uniqueness of bistable wavefronts for the discrete-time recursion defined by the operator $Q_c$. In Section \ref{sec:asymptotic} we study the asymptotic behavior for the solutions of a {related} nonlinear integral equation. There we characterize the behavior of the solutions in terms of the principal eigenvalues of the linearized operator around the stable equilibria. Finally, Section \ref{sec:proofs} is devoted to the proofs of the Theorems \ref{teo:main_1} and \ref{teo:main_2}.
	
	\section{Reduction to an integral formulation}\label{sec:prelim}
	
	From \eqref{GM-model} follows that a traveling wave solution $(c,U=U(a,x))$ must satisfy the following system
	\[
	\begin{dcases}
		(\partial_a - \partial_{xx} - c\partial_x) U(a,x) = -\mu(a) U(a,x), & a>0, ~x\in \R, \\
		U(0,x) = f\left( \int_0^{+\infty} \beta(a) U(a,x) \dd a \right), & x\in \R.
	\end{dcases}
	\]
	Treating the age variable as time, the above equation is a diffusion equation with a constant linear drift $c$, and therefore its solution can be computed in terms of the semigroup of the operator $\partial_{xx}+c\partial_{x}$. Thus, the profile $U$ solves
	\[
	U(a,x) = \Pi(a) \int_\R K(a,x-y+ca) f\left( \int_0^{+\infty} \beta(\sigma) U(\sigma,y) \dd \sigma \right) \dd y,
	\]
	where $K=K(t,x)$ is the heat kernel $K(t,x) = \frac{1}{\sqrt{4\pi t}} e^{-\frac{x^2}{4t}}$. Setting $\phi(x) = \int_0^{+\infty} \beta(a) U(a,x) \dd a$ leads to
	\begin{equation}\label{eq:profile}
		\phi(x) = \int_\R G_c(x-y) f(\phi(y)) \dd y,
	\end{equation}
	with the following asymptotic boundary conditions deduced from \eqref{eq:lim}
	\begin{equation}\label{eq:asymptotics}
		\phi(-\infty)=1 \quad \text{and} \quad \phi(+\infty) = 0,
	\end{equation}
	and where the kernel $G_c$ is defined as
	\begin{equation}\label{eq:Gc}
		G_c(z) = \int_0^{+\infty} \beta(a)\Pi(a)K(a,z+ca) \dd a.
	\end{equation}
	
	The above deduction makes clear the relation between the solutions to the integral equation $\eqref{eq:profile}$ and the wave profiles for the Gurtin-MacCamy model. In fact, every wave profile $U$ provides a solution $\phi$ to \eqref{eq:profile}; whereas given any continuous solution to this equation, the function $U(a,x)$ defined as
	\begin{equation}\label{formula_U}
		U(a,x) = \Pi(a)\int_\R K(a,x-y+ca) f(\phi(y)) \dd y
	\end{equation}
	is a wave profile to the Gurtin-MacCamy model
	
	Therefore, our aim is establishing the existence of a fixed point of the operator $Q_c$ defined by the right-hand side of \eqref{eq:profile}, i.e.,
	\begin{equation}\label{eq:Qc}
		Q_c[u] = G_c \ast f(u),
	\end{equation}
	where $\ast$ denotes the convolution operation for functions on $\R$, and that such a fixed point verifies the asymptotics \eqref{eq:asymptotics}. The main difficulty is that all the operators $Q_c$ may not have such a fixed point, and therefore we must also detect an admissible speed for which this problem admits a solution. This is the classical problem found in the traveling-wave seeking: finding a profile and a speed simultaneously.
	
	\subsection{Properties for the convolution map}
	
	Let us start with some easy observations about the kernel $G_c$.
	\begin{proposition}\label{prop:Gc}
		The kernel $G_c$ defined in \eqref{eq:Gc} possesses the following properties:
		\begin{enumerate}[label= \roman*)]
			\item $G_c\in \BC(\R)\cap \L^1(\R)$, it is a positive function on the whole real line, and is such that
			\begin{equation}\label{bound_Gc}
				\int_\R G_c(z) \dd z = 1 \quad \text{and} \quad \sup_{z\in \R} G_c(z) \leq C\int_0^{+\infty} \frac{e^{-\unmu a}}{\sqrt{4\pi a}} \dd a,
			\end{equation}
			for some positive constant $C$ independent of $c$.
			\item The following formula for the exponential moments of $G_c$ holds
			\begin{equation}\label{eq:exp_mom}
				M_c(\lambda) := \int_\R e^{\lambda z} G_c(z) \dd z = \int_0^{+\infty} \beta(a)\Pi(a) e^{a(\lambda^2-c\lambda)}\dd a = H(\lambda^2 - \lambda c),
			\end{equation}
			with $H$ as in \eqref{eq:laplace_H}. In particular, $M_c$ is defined for all $\lambda$ with $|2\lambda - c| < \sqrt{c^2 + 4\Lambda^*}$.
			\item The map $c\mapsto G_c$ is continuous from $\R$ into $L^1(\R)$.
		\end{enumerate}
	\end{proposition}
	\begin{proof}
		%
		The proofs or (i) and (ii) follow from classical arguments and we omit the proof for concision.  To prove (iii), we
		observe that, by Tonelli's Theorem
		\begin{align*}
			\|G_{\tilde{c}} - G_c\|_{\L^1} \leq \int_0^{+\infty} \beta(a) \Pi(a) \int_\R |K(a, z + \tilde{c} a) - K(a, z + c a)| \dd z \, \dd a.
		\end{align*}
		It is clear that the second integral vanishes for each fixed value of $a$ as $\tilde{c}\to c$ due to the continuity of the shift operator in $\L^1$. In addition, this term is pointwise bounded above by $2$ for each $a$, and therefore $G_{\tilde{c}} \to G_c$ by the Lebesgue's Dominated Convergence Theorem.
	\end{proof}
	
	In order to get shed some light on the domain of definition of $M_c$, we provide a useful characterization for the abscissa of convergence $\Lambda^*$.
	
	\begin{proposition}\label{prop:esti_lambda}
		The abscissa of convergence $\Lambda^*$ of $H$ can be computed by the following formula
		\begin{equation}\label{eq:Lambda_1}
			\Lambda^* = -\limsup_{a\to \infty} \frac{1}{a} \log\left( \int_a^{+\infty} \beta(s) \Pi(s) \dd s \right).
		\end{equation}
		In particular, $\Lambda^* \geq \unmu$, with $\unmu$ as in Assumption \ref{as-1}-(ii). In addition, if $\beta$ is eventually {uniformly positive}, then it holds
		\begin{equation}\label{eq:Lambda_2}
			\Lambda^* = \liminf_{a\to \infty} \frac{1}{a} \int_0^a \mu(l) \dd l.
		\end{equation}
	\end{proposition}
	\begin{proof}
		Formula \eqref{eq:Lambda_1} is an immediate consequence of \cite[Theorem 2.4d p.43]{MR5923}, and the normalization hypothesis \eqref{eq:norm}. Now, we observe that
		\[
		\int_a^{+\infty} \beta(s) \Pi(s)\dd s \leq \beta^* \Pi(a) \int_a^{+\infty} e^{-\int_a^s \mu(l) \dd l} \dd s  \leq \frac{\beta^*}{\unmu} \Pi(a),
		\]
		from where $\Lambda^*\geq \unmu$ is deduced. If we assume in addition that $\beta(a) \geq \beta_*$ for some {$\beta_*>0$} and large values of $a$, then as before we obtain $\int_a^{+\infty} \beta(s) \Pi(s) \dd s \geq \frac{\beta_*}{\mu^*}\Pi(a)$. {The equality \eqref{eq:Lambda_2} follows from this estimate.}  
	\end{proof}
	
	\begin{remark}
		Formula \eqref{eq:Lambda_1} is useful to determine the abscissa of convergence $\Lambda^*$ under biologically meaningful hypotheses. For instance, if $\beta$ is compactly supported, then clearly $\Lambda^*=+\infty$. On the other hand, if $\beta$ has a exponentially decreasing tail, say $\beta(a) = e^{-\alpha a + o(a)}$ with $\alpha>0$, and the population admits an asymptotic mean mortality rate, i.e., the limit in \eqref{eq:Lambda_2} exists and equalts to some value $\mu_\infty$, then we can prove that $\Lambda^* = \alpha + \mu_\infty$.
	\end{remark}
	
	
	Proposition \ref{prop:Gc}, jointly with the monotonicity of $f$, reveals that the set $\CC_1$ is a positively invariant set for the map $Q_c$. The next proposition states some of the basic properties of this map.
	
	\begin{proposition}\label{prop:A1-A4}
		The map $Q_c\colon \CC_1 \to \CC_1$ verifies:
		\begin{enumerate}[label=(\roman*)]
			\item $Q_c$ is translation invariant.
			\item $Q_c$ is continuous with respect to the compact-open topology. That is, if $u\in \CC_1$ and $\{u_n\}_n\subseteq \CC_1$ is such that $u_n \to u$ uniformly on every compact interval, then $Q_c[u_n]\to Q_c[u]$ uniformly on every compact interval.
			\item $Q_c$ is order-preserving, meaning that $u\leq v$ implies $Q_c[u] \leq Q_c[v]$.
			\item $Q_c$ is compact in the compact-open topology. That is, if $\{u_n\}_n\subseteq \CC_1$, then $\{Q_c[u_n]\}_n$ admits a subsequence converging in every compact interval.
		\end{enumerate}
	\end{proposition}
	\begin{proof}
		Statement (i) follows from a change of variables in the convolution integral.\\
		To prove statement (ii), we observe that for arbitrary $\varepsilon>0$ and $R>0$, there is $L$ sufficiently large such that 
		\begin{align*}
			|Q_c[u_n](x) - Q_c[u](x)| &= \left| \int_\R G_c(y) [u_n(x-y) - u(x-y)]\dd y \right| \\
			&\leq \int_{-L}^L G_c(y) |u_n(x-y) - u(x-y)|\dd y + 2\int_{|y|>L} G_c(y) \dd y \\
			&\leq \max_{|y|\leq R + L} |u_n(y) - u(y)| + 2\varepsilon,
		\end{align*}
		for each $x\in[-R,R]$. \\
		Statement  (iii) follows  from the monotonicity of $f$.\\
		To prove statement (iv), we remark that for any $h\in\mathbb{R}$ we have
		\begin{align*}
			|Q_c[u_n](x+h) - Q_c[u_n](x)| &\leq \int_\R f\big(u_n(y)\big) |G_c(x+h-y) - G_c(x-y)| \dd y \\
			&\leq \| G_c(\cdot + h) - G_c \|_{\L^1(\R)},
		\end{align*}
		which implies that $\{Q_c[u_n]\}_n$ is uniformly equicontinuous.  The conclusion follows from the Ascoli-Arzelà Theorem.
	\end{proof}
	
	The translation invariance implies, in particular, that {the set of monotone functions is positively invariant for the map $Q_c$}. Moreover, if $\alpha\in \R$ we have $Q_c[\alpha] = f(\alpha)$, and from the assumptions made on $f$ it is obvious that $Q_c\colon [0,1] \to [0,1]$ has a bistable structure as formulated in hypothesis (A5) in \cite{MR3420507}.
	
	\subsection{The counter-propagation condition}
	\label{subsec:counter-propagation}
	
	We aim to detect an admissible speed $c\in\R$ for which the operator $Q_c$ admits a particular fixed point. As such a fixed point is identified with a traveling wave solution with speed zero for the monotone dynamical system generated by the map $Q_c$, we propose to apply the abstract theory developed in the works of 
	{Fang and Zhao} \cite{MR3420507}, for such  bistable evolution systems. 
	
	More precisely, we plan of applying \cite[Theorem 3.1]{MR3420507} to the operator $Q_c$. Proposition \ref{prop:A1-A4} establishes that $Q_c$ satisfies the hypotheses (A1)--(A4) in \cite{MR3420507}. The bisatbility (A5) is easily checked as $Q_c$ has $u=0$ and $u=1$ are strongly stable solutions of $Q_c[u]=u$, and $Q_c$ has a unique spatially homogeneous fixed point in $(0, 1)$, $\theta$, which is unstable. 
	The only property that still eludes us is the counter-propagation {(A6)} for the monostable subsystems. This condition reads as (cf. \cite[Remark 2.1.]{MR3420507})
	\begin{equation}\label{eq:counter-prop}
		c_+^*(\theta,1;c) + c_{-}^*(0,\theta;c)>0,
	\end{equation}
	where $c_+^*(\theta,1;c)$ denotes the rightward spreading speed for the upper monostable system, and $c_{-}^*(0,\theta;c)$ denotes the leftward spreading speed for the lower monostable system.
	
	Now, taking the coordinate changes $w_\pm(x) = \pm(u-\theta)$ and $F_\pm(u) = \pm( f(\theta \pm u) - \theta )$, \eqref{eq:profile} transforms into a system of the following form
	\begin{equation}\label{eq:mono}
		w(x) = \int_\R G_c(x-y) F(w(y)) \dd y =: \QQ_c[w](x),
	\end{equation}
	where $F$ is a nondecreasing $C^1$-function with bounded derivative that satisfies
	\begin{equation}\label{eq:mono-cond}
		F(0)=0, \quad F(p)=p, \quad F'(0)=f'(\theta)>1, \quad F'(p)<1, \quad F(u)>u \quad \forall u\in(0,p),
	\end{equation}
	for certain $p>0$. Therefore, in this paragraph we will concentrate on estimating the rightward and leftward spreading speeds of the monostable problem \eqref{eq:mono}.
	
	For {any} given $p_0\in (1,F'(0))$ and $p_\infty > \sup_{x\in [0,1]} F'(x)$, we define the linear convolution operators $\underline{\LL}_c[w] = p_0 G_c \ast w$ and $\overline{\LL}_c[w] = p_\infty G_c \ast w$, both acting on the space $\BC(\R)$, and set the notation $\underline{v}_+$ and $\underline{v}_{-}$ for the rightward and leftward spreading speed of the discrete-time semiflow generated by the operator $\underline{\LL}_c$; $\overline{v}_+$ and $\overline{v}_{-}$ are defined similarly (see, e.g., \cite{MR2652175}). Because of the choice made for the values $p_0$ and $p_\infty$, the following relations hold
	\begin{equation}\label{eq:linear_comp}
		\QQ_c[u] \geq \underline{\LL}_c[u] \quad \forall u\in \CC_\delta, \quad \text{and} \quad \QQ_c[u]\leq \overline{\LL}_c[u] \quad \forall u\in \CC_1,
	\end{equation}
	with $\delta$ a positive constant sufficiently small. As in the classical theory for monotone monostable evolution systems \cite{MR2270161,MR2652175,LUI1989269,MR653463}, it is expected that the spreading speeds for the linear operators estimate those of the nonlinear operator. 
	
	\begin{lemma}
		Let $\underline{\LL}_c$, $\overline{\LL}_c$, $\underline{v}_+$, $\underline{v}_{-}$, $\overline{v}_+$ and $\overline{v}_{-}$ be defined as above. Then
		\begin{align*}
			&\overline{v}_+ = \inf_{\lambda>0} \frac{\log(p_\infty M_c(\lambda))}{\lambda}, \quad \overline{v}_{-} = \inf_{\lambda>0} \frac{\log(p_\infty M_c(-\lambda))}{\lambda}, \\
			& \underline{v}_+ = \inf_{\lambda>0} \frac{\log(p_0 M_c(\lambda))}{\lambda} \quad \text{and} \quad \underline{v}_{-} = \inf_{\lambda>0} \frac{\log(p_0 M_c(-\lambda))}{\lambda}
		\end{align*}
		where the infimum is always restricted on the region where the function $M_c$, defined by \eqref{eq:exp_mom}, is finite.
	\end{lemma}
	As this result is classical, we omit the proof but refer to \cite{MR2270161,MR2652175,LUI1989269,MR653463}; and more specifically,  \cite[Lemma 2.9 and Proposition 3.9]{MR2270161} and \cite[Theorem 2.1]{MR2654287}.
	
	Thus, from the linear comparison theorem for monostable discrete-time semiflows (see \cite[Theorem 3.10]{MR2270161} and \cite[Theorem 3.5]{LUI1989269}), it follows that
	\begin{equation}\label{count-estimates}
		\begin{array}{>{\displaystyle}c}
			\inf_{\lambda>0} \frac{\log(p_0 M_c(\lambda))}{\lambda} \leq c_+^*(\theta,1;c) \leq \inf_{\lambda>0} \frac{\log(p_\infty M_c(\lambda))}{\lambda}, \\[10pt]
			\inf_{\lambda>0} \frac{\log(p_0 M_c(-\lambda))}{\lambda}\leq c_{-}^*(0,\theta;c) \leq \inf_{\lambda>0} \frac{\log(p_\infty M_c(-\lambda))}{\lambda},
		\end{array}
	\end{equation}
	with the same convention pointed over the set where the infimum is taken.
	
	\begin{lemma}
		The counter-propagation condition {\eqref{eq:counter-prop} holds} for the bistable semiflow $Q_c$. 
		More precisely, the following estimates hold
		\begin{equation}\label{mono_speeds_1}
			\begin{array}{>{\displaystyle}c}
				c^*_+(\theta,1;c)\geq 2\sqrt{\log f'(\theta)} \int_0^{+\infty} \sqrt{a}\beta(a)\Pi(a) \dd a - c \int_0^{+\infty} a\beta(a)\Pi(a) \dd a, \\[10pt]
				c_-^*(0,\theta;c)\geq 2\sqrt{\log f'(\theta)} \int_0^{+\infty} \sqrt{a}\beta(a)\Pi(a) \dd a + c \int_0^{+\infty} a\beta(a)\Pi(a) \dd a,
			\end{array}
		\end{equation}
		and also, {recalling the definition of $H$ in \eqref{eq:laplace_H},}
		\begin{equation}\label{mono_speeds_2}
			\begin{array}{>{\displaystyle}c}
				\text{if}~ c>0,~ c^*_+(\theta,1;c)\leq \frac{2}{c}\log\left( p_\infty H\left( -\frac{c^2}{4}\right) \right),  \\[10pt]
				\text{if}~ c<0,~ c_-^*(0,\theta;c)\leq -\frac{2}{c}\log\left( p_\infty H\left( -\frac{c^2}{4} \right) \right).
			\end{array}
		\end{equation}
	\end{lemma}
	\begin{proof}
		For the lower bound {on $c^*_+(\theta,1;c)$,} we use \eqref{count-estimates}, \eqref{eq:exp_mom} and Jensen's inequality to get
		\[
		c^*_+(\theta,1;c) \geq \inf_{\lambda>0} \int_0^{+\infty} \beta(a)\Pi(a) \left( a(\lambda -c) + \frac{\log p_0}{\lambda} \right) \dd a,
		\]
		{then minimize} the function $\lambda \mapsto a(\lambda -c) + \frac{\log p_0}{\lambda}$ {for $\lambda>0$, and take the limit $p_0 \uparrow f'(\theta)$}. The upper bound is obtained directly from evaluating the objective function at $\lambda = \frac{c}{2}$. {The estimates on $c^*_-(0,\theta;c)$ are analogous. }
	\end{proof}
	
	\section{Bistable waves for the discrete map $Q_c$} \label{sec:existence}
	
	We start with a regularity result for the nonincreasing traveling wave solutions of the recurrence $u_{n+1} = Q_c[u_n]$.
	
	\begin{proposition}\label{diff}
		Let $u\in C(\R)$ be a nonincreasing bistable wavefront of the discrete-time recursion $u_{n+1} = Q_c[u_n]$, namely,
		\begin{equation}\label{trav_Q}
			Q_c[u](x) = u(x+v) \quad \forall x\in\R, \quad u(-\infty)=1 \quad\text{and}\quad u(+\infty)=0,
		\end{equation}
		for certain $v\in \R$. Then, $u$ is continuously differentiable and strictly decreasing.
	\end{proposition}
	\begin{proof}
		By the monotonicity of $u$ and $f$, the function $F(x) := f(u(x))$ is nonincreasing and connects 1 and 0. Now, {integrating by parts leads to} 
		\[
		u(x) = -\int_\R \left( \int_{-\infty}^y G_c(x-v-z) \dd z \right) \dd F(y) = -\int_\R \int_{x-y}^{+\infty} G_c(z-v) \dd z \, \dd F(y),
		\]
		and {then} $u'(x) = \int_\R G_c(x-v-y) \dd F(y)<0$, since $G_c$ is strictly positive.
	\end{proof}
	
	Now we state a technical lemma inspired by the squeezing technique of Chen \cite{MR1424765}.
	
	\begin{lemma}\label{sub_super_sol}
		Let $u\in C(\R)$ be a nonincreasing bistable wavefront as in \eqref{trav_Q}. Given any $\rho\in( \max\{f'(0),f'(1)\} , 1)$, there exists a small positive number $\delta_0$ (independent of $u$) and a large positive number $\sigma$ (depending on $u$) such that for every $\delta\in(0,\delta_0]$, and every $h^\pm\in \R$, if we define
		\[
		u_n^\pm(x) = u(x+vn+h^\pm \mp \sigma\delta(1-\rho^n)) \pm \delta\rho^n,
		\]
		then $u_n^-$ and $u_n^+$ are sub- and super-solutions of the recursion $u_{n+1} = Q_c[u_n]$, respectively.
	\end{lemma}
	\begin{proof}
		We only focus on the case for $u_n^+$, the other being completely analogous, as well as taking $h^+=0$ since the operator $Q_c$ is translation invariant. We need to prove that $Q_c[u_n^+] \leq u_{n+1}^+$, {or equivalently}
		\[
		Q_c[u(\cdot) + \delta\rho^n](x+vn-\sigma\delta(1-\rho^n)) \leq u(x+v(n+1)-\sigma\delta(1-\rho^{n+1})) + \delta\rho^{n+1}.
		\]
		Then, setting $\xi=x+vn-\sigma\delta(1-\rho^n)$, {the above inequality reduces to}
		\begin{equation}\label{eq:1}
			Q_c[u(\cdot) + \delta\rho^n](\xi) - Q_c[u](\xi -\sigma\delta\rho^n(1-\rho)) \leq \delta\rho^{n+1}, \quad \forall\xi\in\R.
		\end{equation}
		Let us take $\nu<\rho$ sufficiently close to $\rho$. Since $f$ is $C^1$ we can find $\eta>0$ small enough such that $f'(x)\leq \nu$ on $[-2\eta,2\eta]\cup[1-2\eta,1+2\eta]$. As $u$ is a nonincreasing function, the set $I_\eta = \{x\in\R \colon \eta \leq u(x)\leq 1-\eta\}$ is indeed a compact interval, and since $u(-\infty)=1$, we can find $R>0$ large enough such that
		\begin{equation}\label{eq:2}
			u(x-R) - u(x) \geq \frac{\eta}{2} \quad \forall x\in I_\eta.
		\end{equation}
		Thus, we define $\alpha = -\max_{I_\eta+[-R,R]} u' >0$ (e.g. Proposition \ref{diff}). Set now
		\[
		\delta_0 := \frac{\eta}{2} \quad \text{and}\quad \sigma> \frac{1}{\alpha(1-\rho)},
		\]
		and choose any $\delta\in (0,\delta_0]$. Noting that $u(y-\sigma\delta\rho^n(1-\rho))\geq u(y)$ for all $y\in \R$, we have
		\begin{align*}
			\int_{\R\setminus I_\eta} G_c(x-y)\big[f(u(y)+\delta\rho^n) - f(u(y-\sigma\delta\rho^n(1-\rho)))\big] \dd y \leq  \nu \delta\rho^n \int_{\R \setminus I_\eta} G_c(x-y) \dd y,
		\end{align*}
		for each $x\in \R$. To estimate the difference when $y\in I_\eta$ we consider two cases. If $\sigma\delta\rho^n(1-\rho)>R$, then \eqref{eq:2} yields
		\[
		u(y-\sigma\delta\rho^n(1-\rho)) - u(y) -\delta\rho^n \geq \frac{\eta}{2} - \delta\geq 0,
		\]
		and therefore the integral over $I_\eta$ is nonpositive by the monotonicity of $f$. Complementarily, if $\sigma\delta\rho^n(1-\rho)\leq R$, then
		\[
		u(y)+\delta\rho^n - u(y-\sigma\delta\rho^n(1-\rho)) \leq \delta\rho^n -\alpha\sigma\delta\rho^n(1-\rho)<0,
		\]
		and we obtain that the integral over $I_\eta$ is nonpositive as before. Consequently,
		\[
		Q_c[u(\cdot)+\delta\rho^n](x) - Q_c[u](x-\sigma\delta\rho^n(1-\rho)) \leq \nu\delta\rho^n \int_{\R\setminus I_\eta} G_c(x-y) \dd y
		\]
		and \eqref{eq:1} follows because $\nu<\rho$ and the total mass of $G_c$ is 1.
	\end{proof}
	
	Now we can prove the main result of this section.
	
	\begin{theorem}\label{teo:exis_Qc}
		Let Assumption \ref{as-1} holds, and for every $c\in \R$ let $Q_c$ be defined as in \eqref{eq:Qc}. Then, there exists a unique traveling wave solution (modulo translations) $(v,\phi)=(v_c,\phi_c)\in \R \times \CC_1$, with $\phi$ being continuously differentiable and strictly decreasing; that is, the {pair} $(v,\phi)$ satisfies
		$Q_c[\phi_c](x) = \phi_c(x+v),$ 
		{along with} $\phi_c(-\infty)=1$ and $\phi_c(+\infty)=0$. Moreover, the speed $v$ has the following estimates
		\begin{equation}\label{estim_bi_speed}
			-c_{-}^*(0,\theta;c) \leq v \leq c_+^*(\theta,1;c),
		\end{equation}
		where $c_+^*$ and $c_{-}^*$ are the rightward and leftward spreading speeds for the monostable subsystems as defined in {subsection \ref{subsec:counter-propagation}}. 
	\end{theorem}
	\begin{proof}
		The existence of a pair $(v,\phi)$ as claimed is an immediate consequence of the general theory by Fang and Zhao \cite[Theorem 3.1]{MR3420507} as well as from Proposition \ref{diff}. The uniqueness of the speed and the profile follows by the squeezing method thanks to Lemma \ref{sub_super_sol} and the strong comparison principle; see e.g. \cite[Theorem 2.1]{MR1424765}. Finally, the bounds for the speed are consequences of the properties of the spreading speeds of the monostable subsystems; see \cite[Theorem 4]{MR2597125}.
	\end{proof}
	
	\begin{remark}
		Under more restrictive conditions, namely assuming that $f$ has a Lipschitz derivative in $[0,1]$, the uniqueness part follows directly from the abstract results of Zhang and Zhao \cite{MR4305981}, whose ideas are also motivated by the previous work of Chen. We decided to keep the original approach as a manner to keep only the minimal hypothesis needed.
	\end{remark}
	
	\section{Asymptotic behavior}\label{sec:asymptotic}
	
	In this paragraph we will study the asymptotic behavior of the nontrivial solutions for the following integral equation
	\begin{equation}\label{eq:int}
		u(x) = (J \ast f(u))(x), 
	\end{equation}
	where $\ast$ denotes the convolution operation $(J\ast u)(x)=\int_\R J(x-y)u(y) \dd y$, and $f$ is as in Assumption \ref{as-1}. We define the operator $P[u] = J\ast f(u)$ so the above equation can be written as $u = P[u]$. For the kernel we impose the following hypotheses:
	\begin{assumption}\label{as:kernel}\hfill
		\begin{enumerate}[label=(J\arabic*)]
			\item $J\in \L^1(\R)$ is a positive function almost everywhere such that $\int_\R J(x) \dd x = 1$.
			\item The bilateral Laplace transform of $J$, denoted by $\overline{J}(\lambda) = \int_\R e^{-\lambda x} J(x) \dd x$, has the strip of convergence $\{ z \colon \underline{\Lambda}_J < \Re\, z < \overline{\Lambda}_J \}$, with $\underline{\Lambda}_J < 0 < \overline{\Lambda}_J$. Moreover, we assume that 
			\begin{equation}\label{eq:asym_J}
				\overline{J}(\lambda) \longrightarrow +\infty \quad \text{as}~ \lambda\uparrow \overline{\Lambda}_J ~ \text{or}~ \lambda \downarrow \underline{\Lambda}_J.
			\end{equation}
		\end{enumerate}
	\end{assumption}
	
	Throughout this whole section we assume that $u\in C(\R)$ is a nonincreasing solution of \eqref{eq:int} with $0\leq u \leq 1$, $u(-\infty)=1$ and $u(+\infty)=0$. Moreover, we will assume that Assumption \ref{as-2}-(i) holds
	
	The {ideas utilized here are inspired} on the work of Diekmann and Kaper \cite{MR512163}, which ultimately relies on the use of Tauberian results. Hence, we set the notation $\overline{u}(\lambda) = \int_\R e^{-\lambda x} u(x) \dd x$ for the (bilateral) Laplace transform of $u$.
	
	We begin by defining $H_0(\lambda) = 1-f'(0)\overline{J}(\lambda)$ and $H_1(\lambda) = 1-f'(1)\overline{J}(\lambda)$. {It is a classical result that} $H_0$ and $H_1$ are analytic functions in $(\underline{\Lambda}_J, \overline{\Lambda}_J)$ (cf. \cite[Theorem 5a, p. 57]{MR5923}). {Also, as} $J$ is a nonnegative kernel and $f'(0),f(1)<1$ (see Assumption \ref{as-1}-(ii)), $H_0$ and $H_1$ are concave functions with a positive maximum. Therefore, from \eqref{eq:asym_J} we infer that the equations $H_0(\lambda)=0$ and $H_1(\lambda)=0$ have exactly two solutions, one positive and one negative. We define $\lambda_0\in (0,-\underline{\Lambda}_J)$ and $\lambda_1\in (0,\overline{\Lambda}_J)$ as the unique positive solutions of the equations $H_0(-\lambda)=0$ and $H_1(\lambda)=0$, respectively. These zeros are simple because of the concavity.
	
	Now, as a preliminary result, we establish the exponential convergence of the tails of $u$ towards the stable equilibria $1$ and $0$.
	
	\begin{proposition}\label{prop:expo_tails}
		Let $\lambda_0$ and $\lambda_1$ as in the preceding discussion, then 
		\begin{align*}
			u(x) &= O(e^{-\varepsilon x}) \quad \text{as}~x\to\infty \quad \forall \varepsilon\in(0,\lambda_0), \\
			1-u(x) &= O(e^{\varepsilon x}) \quad \text{as}~ x\to-\infty \quad \forall \varepsilon\in (0,\lambda_1)
		\end{align*}
	\end{proposition}
	\begin{proof}
		We will prove only the statement for $x\to \infty$. Given any $\rho\in(f'(0),1)$, from $u(+\infty)=0$ we find $R>0$ large enough so that $f(u(x))\leq \rho u(x)$ for all $x\geq R$. Hence, since $u$ is a solution of \eqref{eq:int} we obtain 
		\begin{equation}\label{eq:E_R}
			u\leq \rho \LL_R[u] + E_R, 
		\end{equation}
		where $\LL_R[u](x) = \int_R^{+\infty} J(x-y) u(y) \dd y$ and $E_R(x) = \int_{-\infty}^R J(x-y) f(u(y)) \dd y$.
		
		Now, if $\varepsilon\in(0,-\underline{\Lambda}_J)$, and since $f(u(x))$ is bounded above by 1, {then an elementary computation shows that $e^{\varepsilon x} E_R(x) \leq e^{\varepsilon R} \overline{J}(-\varepsilon)$ for all $x\geq R$}. In particular, this shows that $E_R$ belongs to the space $C_\varepsilon([R,+\infty))$ of continuous functions with weighted sup-norm $\|u\|_\varepsilon = \sup_{x\geq R} e^{\varepsilon x} |u(x)|$. In a similar fashion, we obtain that $\|\LL_R\|_\varepsilon \leq \overline{J}(-\varepsilon)$. Then, taking $\lambda_\rho$ as the unique solution positive solution of the equation $1-\rho \overline{J}(-\lambda)=0$, it follows that for every $\varepsilon\in(0,\lambda_\rho)$ the operator $\rho\LL_R$ has norm less than 1, and therefore $I- \rho\LL_R$ is invertible. Using the Neumann series for the inverse, jointly with \eqref{eq:E_R}, we get
		\[
		u \leq \sum_{n=0}^{+\infty} (\rho\LL_R)^n[E_R] \in C_\varepsilon([R,+\infty)),
		\]
		what shows that $u(x)=O(e^{-\varepsilon x})$ as $x\to \infty$. Finally, note that $\lambda_\rho \to \lambda_{0}$ as $\rho\downarrow f'(0)$.
	\end{proof}
	
	Proposition \ref{prop:expo_tails} shows that the Laplace transform of $u$ has a region of convergence $(\underline{\Lambda}_u, \overline{\Lambda}_u)$ with $\underline{\Lambda}_u<0< \overline{\Lambda}_u$. Next, we prove that this region is strictly contained in the region of convergence for the Laplace transform of the kernel. 
	
	\begin{proposition}\label{prop:j}
		Using the notation as above for the regions of convergence of $\overline{J}$ and $\overline{u}$, it holds that $\underline{\Lambda}_J<\underline{\Lambda}_u < \overline{\Lambda}_u <\overline{\Lambda}_J$.
	\end{proposition}
	\begin{proof}
		From the assumptions made on $f$ we can find a positive constant $\alpha$ such that $f(x) \geq \alpha x$ for all $x\in[0,1]$. Thus, $u=J \ast f(u) \geq \alpha J \ast u$, and applying the Laplace transform yields $\alpha \overline{J} \overline{u} \leq \overline{u}$, from where $\alpha \overline{J}(\lambda)\le1$ for all $\lambda\in(\underline{\Lambda}_u,0)$, and therefore $\underline{\Lambda}_J \leq \underline{\Lambda}_u$. Furthermore, since $\overline{J}\to +\infty$ as $\lambda\downarrow \underline{\Lambda}_J$, hence $\underline{\Lambda}_J < \underline{\Lambda}_u$. The proof of the fact $\overline{\Lambda}_u < \overline{\Lambda}_J$ is analogous.
	\end{proof}
	
	Next, we define the function $R(x) = f(u(x)) - f'(0)u(x)$, and accordingly with the notation already used, we call $\overline{R}(\lambda)$ to its Laplace transform. Our interest lies in the definiteness of $\overline{R}$ around the points $\underline{\Lambda}_u$ and $\overline{\Lambda}_u$.
	
	\begin{proposition}\label{prop:r}
		The Laplace transform $\overline{R}$ of $R$ is analytic in the strip $\{ (1+\eps)\underline{\Lambda}_u < \Re\, z < (1+\eps)\overline{\Lambda}_u \}$, with $\eps$ given in the part (i) of Assumption \ref{as-2}.
	\end{proposition}
	\begin{proof}
		Let $\lambda\in((1+\varepsilon)\underline{\Lambda}_u, 0)$. From Assumption \ref{as-2}-(i), there is a sufficiently large $L$ such that $|R(x)| \leq C u(x)^{1+\eps}$ for all $x\geq L$ and some positive constant $C$. Also, since $u$ is nonincreasing, for every $\sigma\in(\underline{\Lambda}_u, 0)$ we have $u(x)\int_{-\infty}^x e^{-\sigma y} \dd y \leq \int_{-\infty}^x e^{-\sigma y}u(y) \dd t \leq \overline{u}(\sigma)$, and therefore $u(x) \leq C_\sigma e^{\sigma x}$ for every $x\in \R$. Now, observe that since $-\lambda$ is positive, it suffices to show that the integral defining $\overline{R}(\lambda)$ is finite over the interval $[L,+\infty)$. Based on what has been done before, it follows that
		\[
		\int_L^{+\infty} e^{-\lambda x} R(x) \dd x \leq C \int_L^{+\infty} e^{-\lambda x}e^{(1+\eps)\sigma x} \dd x = C \int_L^{+\infty} e^{-(\lambda -\sigma(1+\varepsilon))x} \dd x,
		\]
		and the last term is finite by choosing $\sigma$ sufficiently close to $\underline{\Lambda}_u$, this completes the proof.
	\end{proof}
	
	We now state the main result of this section.
	\begin{theorem}\label{teo:asymptotic}
		Let the Assumption \ref{as:kernel}  be satisfied, and let $u\in C(\R)$ be a nonincreasing solution of \eqref{eq:int} with $u(-\infty)=1$ and $u(+\infty)=0$. Then, there are positive constants $C_0$ and $C_1$ such that
		\begin{equation}\label{asym}
			u(x) \sim C_0 e^{-\lambda_0 x} \quad \text{as}~x\to+\infty \quad \text{and}\quad  1-u(x) \sim C_1 e^{\lambda_1 x} \quad \text{as}~ x\to - \infty,
		\end{equation}
		with $\lambda_0$ and $\lambda_1$ defined as in the beginning of the section.
	\end{theorem}
	\begin{proof}
		We will only prove the first part. We start rewriting \eqref{eq:int} as $u = f'(0) J \ast u + J\ast R$, and taking the Laplace transform leads to
		\begin{equation}\label{eq:laplace}
			H_0(\lambda)\overline{u}(\lambda) = \overline{J}(\lambda)\overline{R}(\lambda).
		\end{equation}
		From equation \eqref{eq:laplace} and the fact that $\overline{u}$ is singular at $\underline{\Lambda}_u$ (see \cite[Theorem 5b, p. 58]{MR5923}), we infer that $H_0(\underline{\Lambda}_u)=0$, what jointly with the definition of $\lambda_0$ follows that $\underline{\Lambda}_u = -\lambda_0$. A similar analysis yields that $\overline{R}(-\lambda_0)\neq 0$ since this point is a simple zero of $H_0$. Observe also that $H_0$ has no more zeros on the line $\Re\, z= -\lambda_0$. In fact, if $\omega\in \R$ is such that $H_0(-\lambda_0 + i\omega)=0$, equating the real parts gives
		\[
		1-f'(0)\int_\R e^{-\lambda_0 x} J(x)\cos(\omega x) \dd x = 0 = 1- f'(0) \int_\R e^{-\lambda_0 x} J(x) \dd x,
		\]
		and this implies that $\omega=0$ since $J>0$ a.e.
		
		{Owing to the above observations, the conclusion follows as in the proof of Theorem 6.2 in \cite{MR512163}, mainly by an application of the Ikehara's Tauberian Theorem (see \cite[Proposition 2.3]{MR2052422}, as well as Chapter V, section 17 in \cite{MR5923}). Moreover, the constant $C_0$ in \eqref{asym} is given by the residue of $\overline{u}$ at $-\lambda_0$, that is, $C_0 = \frac{\overline{J}(-\lambda_0) \overline{R}(-\lambda_0)}{f'(0) \overline{J}'(-\lambda_0)}$.}
	\end{proof}
	
	\section{Proofs of Theorems \ref{teo:main_1} and \ref{teo:main_2}} \label{sec:proofs}
	
	Owing to Theorem \ref{teo:exis_Qc}, the function $c\mapsto v_c$ is well defined from $\R$ into $\R$. Thus, we seek for some particular value of $c$ at which $v_c=0$, {which we obtain by the intermediate value theorem.} 
	
	
	\begin{proof}[Proof of Theorem \ref{teo:main_1}]
		We have to show that for some particular value of $c=c^*$, the associated speed $v^*$ equals to zero, yielding a standing wave of the discrete-time recurrence equation defined by $Q_c$ and hence a traveling wave solution of the Gurtin-MacCamy model. First, from \eqref{mono_speeds_2} and \eqref{estim_bi_speed} we observe that for large values $c\gg 1$ the associated speed $v_c$ becomes negative; alternatively, for large negative values $c\ll -1$, the speed becomes positive. {Since the speed $v_c$ is unique for each $c\in\mathbb{R}$, the continuity of the map $c\mapsto v_c$ follows from classical arguments, as the limit of any convergent sequence is uniquely identified}. Thus, an application of the intermediate value theorem leads to the existence of some $c^*\in \R$ such that the associated wave profile $\phi^*$ is a fixed point of $Q_{c^*}$.
		
		As was discussed in Section \ref{sec:prelim}, such a fixed point can be used to construct a traveling wave solution of \eqref{GM-model} by means of the formula
		\[
		U(a,x) = \Pi(a) \int_\R K(a,x+c^* a -y) f(\phi^*(y)) \dd y.
		\]
		From the monotonicity of $\phi^*$ it is clear that $U$ is monotone in the space variable. Furthermore, every bistable traveling wave solution $(c,U)$ of \eqref{GM-model} provides a bistable wave $\phi$ to the equation \eqref{eq:Qc}, and then Theorem \ref{teo:exis_Qc} implies that it should be strictly decreasing.
	\end{proof}
	
	\begin{proof}[Proof of Theorem \ref{teo:main_2}]
		First we assume that the invasion condition $\int_0^1 (f(x) - x) \dd x >0$ is satisfied. Let $c\in \R$ and let $\phi$ be any nonincreasing solution of \eqref{eq:Qc}, we aim to prove that $c>0$. To accomplish this, define $F(x) = f(\phi(x))$ and observe that $F$ is a nonincreasing function. Since $\phi$ solves \eqref{eq:Qc}, Tonelli's Theorem gives
		\begin{equation}\label{eq:6}
			\int_\R \phi(x) F'(x) \dd x = \int_\R G_c(y) \int_\R F(x-y) F'(x) \dd x\, \dd y.
		\end{equation}
		Set $\FF(y) = \int_\R F(x-y) F'(x) \dd x$. The following properties of the function $\FF$ are established from direct integration and integration by parts
		\[
		\FF(0) = -\frac{1}{2} \quad \text{and} \quad \FF(y) + \FF(-y) = -1.
		\]
		In particular, the last one shows that the function $\mathscr{F}(y) = \FF(y) + \frac{1}{2}$ is odd. Also,  $F$ being nonincreasing implies that $\FF$ is nonincreasing as well, and therefore $\mathscr{F}(y) \leq 0$ for all nonnegative $y$. Then, from \eqref{eq:6} and the unit total mass of $G_c$ it follows that
		\[
		\int_\R \phi(x)F'(x) \dd x = -\frac{1}{2} + \int_\R G_c(y) \mathscr{F}(y) \dd y = -\frac{1}{2} + \int_0^{+\infty} (G_c(y) - G_c(-y)) \mathscr{F}(y) \dd y.
		\]
		
		On the other hand, a direct computation shows that $\int_\R\phi(x) F'(x) \dd x = -\frac{1}{2} + \int_0^1 (f(x) - x) \dd x$, what jointly with the above relation gives
		\[
		\int_0^{+\infty} (G_c(x) - G_c(-x)) \mathscr{F}(x) \dd x = \int_0^1 (f(x)-x) \dd x.
		\]
		To estimate the difference in the first integral note that
		\[
		G_c(y) - G_c(-y) = \int_0^{+\infty} \beta(a)\Pi(a)[ K(a,y+ca) - K(a,y-ca)] \dd a,
		\]
		what this implies that $c>0$. Indeed, if not, the monotonicity of $K$ on $\R_{-}$ and $\R_+$, as well as its symmetric properties, shows that $G_c(y) - G_c(-y)\geq 0$ for all $y\geq 0$. Then a contradiction is obtained because $\mathscr{F}$ is nonpositive.
		
		The statement (ii) follows from Theorem \ref{teo:asymptotic}, since any traveling wave $(c,U)$ defines a solution to \eqref{eq:Qc} by $\phi(x) = \int_0^{+\infty} \beta(a) U(a,x) \dd a$, and to translate the properties of $\phi$ to $U$ which we use the representation formula \eqref{formula_U}.
	\end{proof}

	\section*{Acknowledgments}

	\noindent H. K. is partially supported by the NSFC grant (12622130, 12301259 and 12371169). F. H. is supported by Agencia Nacional de Investigación y Desarrollo (ANID) through the grant Beca de Doctorado Nacional 2024-21240616 and by Région Normandie through the RIN50 project "DAMES". 

	\bibliographystyle{amsplain}
	\bibliography{biblio.bib}
	
\end{document}